\documentclass{article}
\usepackage[utf8]{inputenc}
\usepackage{indentfirst}
\usepackage{graphicx,tikz}

\usepackage{amsthm}
\newtheorem{problem}{Question}%[section]

\newtheorem{theorem}[problem]{Theorem}
\newtheorem{lemma}[problem]{Lemma}
\newtheorem{conjecture}[problem]{Conjecture}

\title{On Intersecting Polygons}
\author{Kada K Williams}
\date{March 6, 2023}

\begin{document}

\maketitle

\begin{abstract}
Consider two regions in the plane, bounded by an $n$-gon and an $m$-gon, respectively. At most how many connected components can there be in their intersection? This question was asked by Croft. We answer this asymptotically, proving the bounds 
$$\left\lfloor \frac{m}{2}\right\rfloor \cdot \left\lfloor \frac{n}{2}\right\rfloor\le f(n,m)\le \left\lfloor \frac{m}{2}\right\rfloor \cdot \frac{n}{2} + \frac{m}{2}
$$
where $f(n,m)$ denotes the maximal number of components and $m\le n$. Furthermore, we give an exact answer to the related question of finding the maximal number of components if the $m$-gon is required to be convex: $\left \lfloor \frac{m+n-2}{2}\right\rfloor$ if $n\ge m+2$ and $n-2$ otherwise.
\end{abstract}

\section{Introduction}

Croft \cite{Cro} asked the following question.

\begin{problem}\label{original}
What is the maximal number of connected components in the intersection of the polygonal regions determined by an $n$-gon and an $m$-gon?
\end{problem}

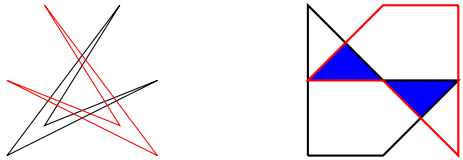
\begin{figure}[ht]
\centering
\begin{tikzpicture}
\draw (-4,0) -- (-2,1) -- (-3.5,0.4) -- (-2.5,2) -- (-4,0);
\draw[red] (-2,0) -- (-4,1) -- (-2.5,0.4) -- (-3.5,2) -- (-2,0);
\fill[blue] (1,1) -- (1.5,0.5) -- (2,1) -- (1,1);
\fill[blue] (0,1) -- (1,1) -- (0.5,1.5) -- (0,1);
\draw[thick] (0,0) -- (1,0) -- (2,1) -- (1,1) -- (0,2) -- (0,0);
\draw[thick,red] (2,2) -- (1,2) -- (0,1) -- (1,1) -- (2,0) -- (2,2);
%\draw[thin,dashed] (0,0) -- (1,1);
%\draw[thin,dashed] (1,0) -- (1,1);
%\draw[red,thin,dashed] (2,0) -- (1,2);
%\draw[red,thin,dashed] (1,1) -- (1,2);
\end{tikzpicture} 
\caption{intersecting polygons}
\label{examples}
\end{figure}

As usual, a \textit{polygon} is a piecewise linear simple closed curve in the plane. Since all polygons admit a triangulation, the $n$-gonal region can be partitioned into $n-2$ triangles, the $m$-gonal region into $m-2$. Now the intersection of convex regions is convex, every triangle is convex, and every convex region is either path-connected or empty. It follows that the answer is at most $(m-2)(n-2)$, hence finite. Furthermore, supposing that the $m$-gon is convex, there are at most $n-2$ components.

\begin{problem}\label{related}
What is the maximal number of connected components in the intersection of a region bounded by an $n$-gon and a convex $m$-gon?
\end{problem}

In Figure \ref{examples}, the yellow regions are in different connected components if the polygons are viewed as open, whereas if the polygons are closed, there is only one connected component. Yet as we will see, the maximal number of connected components in Questions \ref{original} and \ref{related} does not depend on whether the polygonal regions are open or closed. Essentially, we will deform the polygons slightly if these open overlaps have a boundary point in common. We say that an {\it overlap} is a connected component in the intersection of an open $n$-gon and an open $m$-gon.

We claim that every overlap is a polygon, and that the boundary of an overlap consists of segments contained by the perimeter of the $n$-gon or the $m$-gon. To this end, let us first consider the convex polygons formed by intersecting $n-2$ triangles with $m-2$ triangles in their respective triangulations. The overlaps arise when we glue these polygons along the inner diagonals, and so are polygons as well. Their boundary segments are contained not by inner diagonals, but by sides of the $n$-gon or $m$-gon.

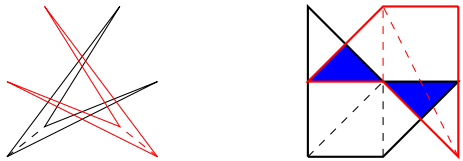
\begin{figure}[ht]%, scale=0.5]
\centering
\begin{tikzpicture}
\draw (-4,0) -- (-2,1) -- (-3.5,0.4) -- (-2.5,2) -- (-4,0);
\draw[red] (-2,0) -- (-4,1) -- (-2.5,0.4) -- (-3.5,2) -- (-2,0);
\draw[thin,dashed] (-4,0) -- (-3.5,0.4);
\draw[red,thin,dashed] (-2,0) -- (-2.5,0.4);
\fill[blue] (1,1) -- (1.5,0.5) -- (2,1) -- (1,1);
\fill[blue] (0,1) -- (1,1) -- (0.5,1.5) -- (0,1);
\draw[thick] (0,0) -- (1,0) -- (2,1) -- (1,1) -- (0,2) -- (0,0);
\draw[thick,red] (2,2) -- (1,2) -- (0,1) -- (1,1) -- (2,0) -- (2,2);
\draw[thin,dashed] (0,0) -- (1,1);
\draw[thin,dashed] (1,0) -- (1,1);
\draw[red,thin,dashed] (2,0) -- (1,2);
\draw[red,thin,dashed] (1,1) -- (1,2);
\end{tikzpicture} 
\caption{overlaps}
\label{examples}
\end{figure}

When intersecting a closed $n$-gon and a closed $m$-gon, it is possible that the boundaries of two overlaps meet, or that a one-dimensional connected component occurs from a vertex or side of one polygon tangent to the side of another. Observe that this is only possible if three side lines of the $n$-gon and $m$-gon pass through a point. We resolve this as follows.

\begin{figure}[ht]
\centering
\begin{tikzpicture}
\draw (0,-0.5) -- (0.5,0) -- (2.5,0) -- (3,-0.5);
\draw[thin,dashed] (0,0) -- (3,0);
\draw[->] (-1,0) -- (-1,0.5);
%\draw (1,1) -- (1.3,1.3);
%\draw (1,1.3) -- (1.3,1);
\draw[thin,red] (1.8,1) -- (2,0) -- (2.4,1) -- (2.4,-0.5) -- (2.6,1.3);

\draw (5,-0.5) -- (6,0.5) -- (7,0.5) -- (8,-0.5);
\draw[thin,dashed] (5,0.5) -- (8,0.5);
%\draw (6,1) -- (6.3,1.3);
%\draw (6,1.3) -- (6.3,1);
\draw[thin,red] (6.8,1) -- (7,0) -- (7.4,1) -- (7.4,-0.5) -- (7.6,1.3);
\end{tikzpicture}
\caption{deformation, before and after}
\label{examples}
\end{figure}
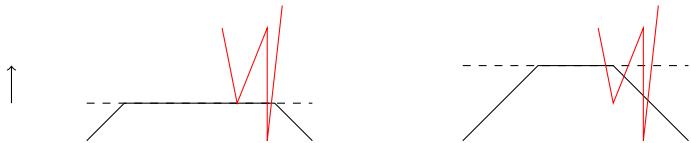

Given a side, let us translate its line by a perpendicular vector such that we increase the width of any component it contains, we do not let the line move across the finitely many points where two side lines meet, and the vector has smaller length than the positive width overlaps. Iterating this, the number of times three side lines concur can be decreased to zero, while the number of components does not decrease. Therefore, in our investigations, we may assume that no three side lines pass through a point, and that the polygons are open.

Let $f(n,m)$ denote the maximal number of overlaps (equivalently, of connected components) between an $n$-gonal and an $m$-gonal region. Our main result is the following.

\begin{theorem}\label{gen}
$\left\lfloor \frac{m}{2}\right\rfloor \cdot \left\lfloor \frac{n}{2}\right\rfloor\le f(n,m)\le \left\lfloor \frac{m}{2}\right\rfloor \cdot \frac{n}{2} + \frac{m}{2}
$
\end{theorem}

Furthermore, we give an exact answer to the convex version of the question.

\begin{theorem}\label{convex} Let us intersect an $n$-gonal region with a convex $m$-gonal region. The maximal number of overlaps is $\left\lfloor \frac{m+n-2}{2}\right\rfloor$ if $n\ge m+2$ and $n-2$ otherwise.
\end{theorem}

\begin{proof}[Proof of Theorem \ref{gen}, lower bound] Intersecting $\left\lfloor \frac{m}{2}\right\rfloor$ 'teeth' with $\left\lfloor \frac{n}{2}\right\rfloor$ 'teeth' as in Figure \ref{saw} yields $\left\lfloor \frac{m}{2}\right\rfloor \cdot \left\lfloor \frac{n}{2}\right\rfloor$ overlaps. Clearly, the direction of the end segments is such that they meet. Possibly with an additional side in the case when $m$ is odd or when $n$ is odd, this creates an $n$-gon and an $m$-gon. \end{proof}

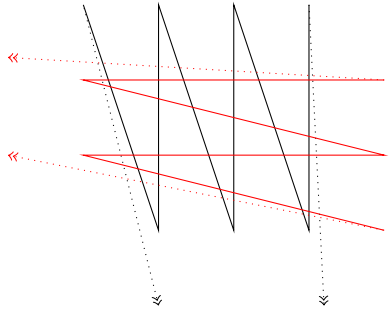
\begin{figure}[ht]
\centering
\begin{tikzpicture}
\draw (0,5) -- (1,2) -- (1,5) -- (2,2) -- (2,5) -- (3,2) -- (3,5);
\draw[<<-, dotted] (1,1) -- (0,5);
\draw[->>, dotted] (3,5) -- (3.2,1);
\draw[thin,red] (4,2) -- (0,3) -- (4,3) -- (0,4) -- (4,4);
\draw[<<-,thin,red, dotted] (-1,3) -- (4,2);
\draw[->>,thin,red, dotted] (4,4) -- (-1,4.3);
\end{tikzpicture} 
\caption{double dentures}
\label{saw}
\end{figure}

Proving the upper bound in Theorem \ref{gen} involves the following lemma. We shall use standard notation -- see e.g. \cite{BMS}.

\begin{lemma}\label{alt}
If an overlap does not contain a vertex of the $m$-gon or $n$-gon, then it has an even number of sides.
\end{lemma}

\section{Three special cases}

\begin{figure}[ht]
\centering
\begin{tikzpicture}
\draw (0,0) -- (-4,2) -- (-1,1) -- (1,1) -- (4,2) -- (0,0);
\draw[thin,red] (-2.8,1.5) -- (0,0.9) -- (2.8,1.5) -- (-2.8,1.5);
\draw[thin, dashed] (0,0) -- (-1,1);
\draw[thin, dashed] (0,0) -- (1,1);
\end{tikzpicture} 
\caption{$f(5,3)=3$}
\label{3,5}
\end{figure}

In the case $m=3$ and $n=5$, by triangulating the pentagon, it is clear that at most $3$ overlaps emerge. Moreover, as many as $3$ overlaps can be attained. 

In the case $m=4$ and $n=4$, the number of overlaps is $\le 2\cdot 2=4$, because each quadrilateral is made up of two triangles. This many overlaps can be attained, see Figure \ref{4,4}.

\begin{figure}[ht]
\centering
\begin{tikzpicture}
\draw (0,0) -- (-2,4) -- (0,1) -- (2,4) -- (0,0);
\draw[thin,red] (3,2) -- (-2,1) -- (1.5,2) -- (-2,3) -- (3,2);
\draw[thin, dashed] (0,0) -- (0,1);
\draw[thin, red, dashed] (3,2) -- (1.5,2);
\end{tikzpicture} 
\caption{$f(4,4)=4$}
\label{4,4}
\end{figure}

In the case $m=5$ and $n=5$, by tweaking the construction for $n=3$, we can achieve $5$ overlaps. On the other hand, our proof that no more than $5$ can be had is more subtle. First, we note that a triangle can only form $3$ overlaps with a pentagon if the pentagon can be partitioned into no less than $3$ convex regions. This is only possible if the pentagon has three sides proceeding along a concave arc, met by the triangle according to Figure \ref{3,5}. 

If neither pentagon has this shape, then at most $2\cdot 2=4$ overlaps are possible. If one pentagon has this shape, then of its three triangle parts, the outer two are intersected at most twice, or else the other pentagon is determined. In particular, if an outer triangle meets the other pentagon in three regions, then two of these regions border an inner diagonal, thus forming no more than $\frac12$ of an overlap (except for an option where these three are all the overlaps). As for the central triangle, it contains at most $1+\frac12+\frac13$ overlaps (except for an option where the other pentagon meets no inner diagonal, with two overlaps from the central triangle and at most one from each outer one). In total, there are certainly fewer than $6$ overlaps, whichever case it is. Therefore, $f(5,5)=5$.

\begin{figure}[ht]
\centering
\begin{tikzpicture}
\draw (0,0) -- (-5,4) -- (-1,1) -- (1,1) -- (5,4) -- (0,0);
\draw[thin,red] (-3.5,1.7) -- (0,0.85) -- (3.7,1.7) -- (-5,3) -- (3.1,1.7) -- (-3.5,1.7);
\end{tikzpicture} 
\caption{construction for $f(5,5)=5$}
\label{5,5}
\end{figure}
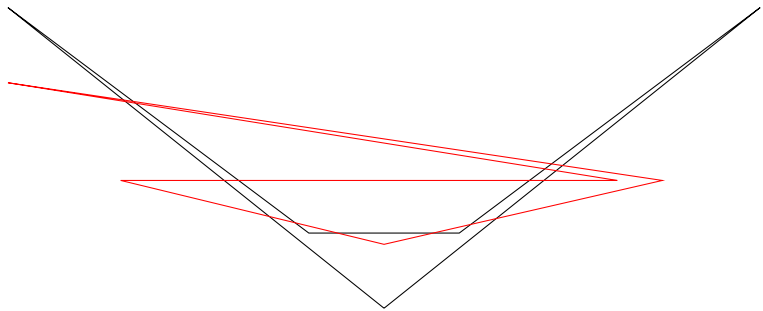

\section{The proofs}

Let us show that if an overlap does not contain a vertex of the $m$-gon or $n$-gon, then it has an even number of sides.

\begin{proof}[Proof of Lemma \ref{alt}]
The key is to inspect two consecutive sides of an overlap. Recall that any side of an overlap is contained by the boundary of the $m$-gon or $n$-gon. If two consecutive sides belong to the $m$-gon, then the vertex where they meet must be a vertex of the $m$-gon, and similarly for the $n$-gon. Therefore, provided that an overlap does not contain such a vertex, its consecutive sides belong only to the $m$-gon or only to the $n$-gon, alternating between the two. Hence, this overlap has an even number of sides along its boundary.
\end{proof}

\begin{proof}[Proof of Theorem \ref{gen}, upper bound] Let there be given an $m$-gon and an $n$-gon with $F$ overlaps. Of these, let there be $F_k$ many with exactly $k$ sides belonging only to the $n$-gon. If an overlap has two consecutive sides on the $m$-gon, then they meet at a vertex of the $m$-gon. We may assume overlaps are disjoint, so each vertex of the $m$-gon occurs at most once this way. We deduce that $3F_0+F_1\le m$, as surely all vertices are such if all sides belong to the $m$-gon, and surely at least one vertex is such if there is just one exception.

The essential idea is to double-count the overlap side segments on the $n$-gon. Looking at one single side of the $n$-gon, it contains at most $m$ points of intersection with the $m$-gon, and so contains at most $\left\lfloor \frac{m}{2}\right\rfloor$ segments bounding some overlap. This amounts to $n\cdot \left\lfloor \frac{m}{2}\right\rfloor$ segments in total. On the other hand, there are $k$ segments on $F_k$ many overlaps. Therefore,
$$\sum_k kF_k\le n\cdot \left\lfloor \frac{m}{2}\right\rfloor.$$
Adding to this the inequality $3F_0+F_1\le m$, the coefficient of each $F_k$ on the left-hand side is at least $2$. Thus, 
$$2F\le n\cdot \left\lfloor \frac{m}{2}\right\rfloor+m.$$
Dividing by $2$ implies the upper bound $f(n,m)\le \left\lfloor \frac{m}{2}\right\rfloor \cdot \frac{n}{2} + \frac{m}{2}$.
\end{proof}

\begin{figure}[ht]
\centering
\begin{tikzpicture}
\draw (0,1) -- (-4,4) -- (-1,2) -- (1,2) -- (1,3) -- (1.2,2.2) -- (4,4) -- (0,1);
\draw[thin,red] (-2.8,3) -- (0,1.8) -- (2.8,3) -- (-2.8,3);
\end{tikzpicture} 
\caption{appending a tooth}
\label{tooth}
\end{figure}
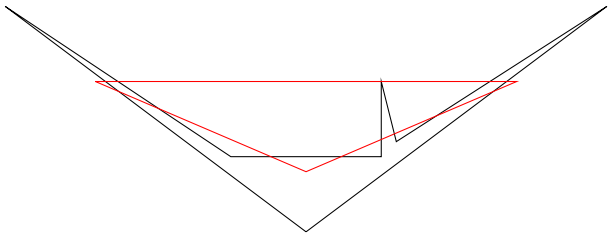

Next, let us see how to maximise the number of overlaps with a convex $m$-gon.

\begin{proof}[Proof of Theorem \ref{convex}] With regard to Figure \ref{3,5}, if a concave arc of $m$ sides is formed over a $V$ shape, the resulting $(m+2)$-gon is intersected in $m$ overlaps by a convex $m$-gon with vertices near the midpoints of these sides. Thus, if $n=m+2$, there can be $n-2=m$ overlaps. If $m$ is increased, there still can be $n-2$ overlaps. If $n$ is increased, then at the cost of two sides, a 'tooth' can be appended as in Figure \ref{tooth}, causing $\left\lfloor \frac{n-(m+2)}{2} \right\rfloor$ more overlaps.

As the $n$-gon can be partitioned into $n-2$ triangles, each of which overlap the convex $m$-gon at most once, there can be no more than $n-2$ overlaps in the case $n<m+2$. Let us further prove that there can be no more than $m+\left\lfloor \frac{n-(m+2)}{2} \right\rfloor=\left\lfloor \frac{m+n-2}{2}\right\rfloor$ overlaps in the case $n\ge m+2$.

Following the same notation as in the previous proof, the number of overlap side segments on the $n$-gon is $\sum kF_k$, where $3F_0+F_1\le m$. However, in this problem, each side of the $n$-gon creates at most one segment bounding some overlap, and so $\sum kF_k\le n$. Adding these yields
$$2F\le m+n.$$
Our strategy is to improve this bound by $2$. Let us regard the sides of the $n$-gon which belong to its convex hull, of which there are at least $2$. Should these fail to meet the $m$-gon, clearly $\sum kF_k\le n-2$. If the $m$-gon has two vertices outside the $n$-gon, $3F_0+F_1\le m-2$. Finally, if the $m$-gon has but one vertex outside the $n$-gon, then the convex hull of the $n$-gon is only met by two consecutive sides of the $m$-gon, so only one overlap is bounded with the convex hull. In this case, $3F_0+F_1\le m-1$, but also $\sum kF_k\le n-1$. Therefore,
$$2F\le m+n-2.$$
It follows that no more than $\left\lfloor \frac{m+n-2}{2}\right\rfloor$ overlaps can be attained.
 \end{proof}

\section{Concluding remarks}

Despite our efforts to approximate $f(n,m)$, Problem \ref{original} remains open. As we saw in our proof that $f(5,5)=5$, if an overlap does necessarily involve a vertex of the $m$-gon, then it takes up considerable space. Our upper bound appears to be wasteful in general. Therefore, we may conjecture the following.

\begin{conjecture}
The answer to Problem \ref{original} is $f(n,m)=\left\lfloor \frac n2\right\rfloor \cdot \left\lfloor \frac m2 \right\rfloor+\epsilon$, where $\epsilon=1$ for odd values of $n$ and $m$, and $\epsilon=0$ otherwise.
\end{conjecture}

The extra term $\epsilon$ is needed because in Figure \ref{saw}, a new overlap is made when swapping the lowest edge for a triangle reaching a new edge, as in Figure \ref{3,5}.

\medskip

\textbf{Acknowledgement.} The author would like to thank Imre Leader for helpful discussions and guidance in writing this paper.

\textsc{Department of Pure Mathematics and Mathematical Statistics, University of Cambridge, Wilberforce Road, Cambridge CB3 0WB.} \\ \\
\textit{E-mail address:} \texttt{kkw25@cam.ac.uk}

\end{document}